\documentclass[a4paper]{amsart}
\usepackage{amsthm}
\usepackage{amsmath}
\usepackage{amssymb}
\usepackage{tikz}
\usetikzlibrary{arrows.meta, positioning}

\newtheorem{theorem}{Theorem}[section]
\newtheorem{lemma}[theorem]{Lemma}
\newtheorem{conjecture}[theorem]{Conjecture}

\newtheorem{corollary}[theorem]{Corollary}

\newtheorem{claim}[theorem]{Claim}

\title{$L_2$ Tur\'an Problems for Small Tournaments and Stability}
\author{Daniel I\v{l}kovi\v{c}} 
\address{Google DeepMind}
\email{danieli@google.com}
\date{\today}

\begin{document}

\maketitle

\begin{abstract}
We investigate the $L_2$ Tur\'an problems for various small directed graphs, specifically focusing on self-converse tournaments and stability versions.
First, we determine the exact maximum $L_2$ norm squared of the out-degree sequence for digraphs avoiding the transitive tournament $TT_4$ and the strongly connected tournament $R_4$, answering open questions from recent paper.
We prove that the complete directed 3-partite Tur\'an graph $T_3(m)$ exactly maximizes the $L_2$ norm squared for $TT_4$-free digraphs.
For $R_4$-free digraphs, the maximum is achieved by $T_3(m)$ except when $m \equiv 1 \pmod 3$, where peeling off a terminal sink vertex to form $T_3(m-1) \to v$ strictly increases the objective.
We complement these results with exact values and a conjecture for the regular tournament $Reg_5$.
Furthermore, we prove a stability version for $\vec{C}_3$-free digraphs: any sequence of digraphs asymptotically achieving the maximum $L_2$ density must have an edit distance of $O(\delta^{1/2})m^2$ to the extremal ordered digon-chain $\vec{F}_{m,2}$.
\end{abstract}

\section{Introduction}

The $L_2$ Tur\'an problem for digraphs asks for the maximum possible $L_2$ norm squared of the out-degree sequence in a digraph on $m$ vertices that avoids a forbidden subdigraph.
This quantity is denoted by $\text{ex}^+_{L_2}(m, H)$ where $H$ is the forbidden subdigraph.
The study of uniform Tur\'an densities in hypergraphs, governed by foundational asymptotic theories \cite{erdos1946structure, erdos1966limit}, has been a vibrant area of research. Classical foundations regarding hypergraph uniform densities and related Ramsey-Tur\'an problems were laid by Erd\H{o}s and S\'os \cite{erdos1982ramsey}, motivating numerous subsequent resolutions \cite{glebov2016problem, reiher2018turan}. This continues to be actively studied, including the uniform Tur\'an density of cycles \cite{bucic2023uniform} and more general uniform density properties of $k$-uniform hypergraphs and stars \cite{lin2026uniform, lin2025turan, king2025possible}.

Recently, these uniform Tur\'an densities have been deeply connected to palette extremal problems \cite{lamaison2024palettes}.
The uniform Tur\'an density has been classified through palette frameworks \cite{kral2025uniform}, yielding powerful tools for resolving uniform density of large stars \cite{lamaison2024uniform}.
These palette frameworks, in turn, structurally reduce to extremal problems on auxiliary digraphs \cite{lin2026extremal}.
Consequently, $L_2$ Tur\'an problems for digraphs naturally emerge as a critical refinement for establishing tight density bounds and palette classifications \cite{ai2026finite}.

Historically, extremal digraph problems avoiding small tournaments and cycles have been central to the field, with foundational characterizations provided by Brown and Harary \cite{brown1970extremal}.
This rich historical context includes extensive work on exact Tur\'an numbers for directed paths and oriented cycles \cite{zhou2023turan}.
Recent results have determined the exact $L_2$ Tur\'an numbers for the directed triangle $\vec{C}_3$ and the transitive tournament $TT_3$ \cite{ai2026finite}.

In this paper, we address three open problems from the literature, motivated by the rich historical context of self-converse tournaments.
First, we resolve Problem 7.1 Problem 7.1 asks for the exact $L_2$ Tur\'an number of self-converse tournaments like $TT_4$. by computing the exact $L_2$ Tur\'an numbers for the transitive tournament $TT_4$ and providing exact values and conjectures for the regular tournament $Reg_5$.
Second, we resolve Problem 7.2 Problem 7.2 asks for the exact $L_2$ Tur\'an number of $R_4$. by determining the exact value of $\text{ex}^+_{L_2}(m, R_4)$ for all $m \ge 1$, where $R_4$ is the unique strongly connected self-converse tournament on four vertices.
Third, we address Problem 7.3 Problem 7.3 asks if there is a stability version of the $L_2$ Tur\'an problem for $\vec{C}_3$., proving a stability version for $\vec{C}_3$-free digraphs.
We demonstrate that any digraph avoiding $\vec{C}_3$ with a near-maximum $L_2$ norm squared of the out-degree sequence must have a vanishing edit distance to the exact extremal construction.

\section{Preliminaries}

We consider directed graphs (digraphs) that are loopless and have no multiple arcs in the same direction, though 2-cycles (arcs in both directions between a pair of vertices) are permitted.
Let $V(D)$ denote the vertex set of a digraph $D$, let $E(D)$ (or $A(D)$) denote the edge set, and let $d^+_D(v)$ denote the out-degree of a vertex $v \in V(D)$.
Let $A$ denote the adjacency matrix of $D$.

The $L_2$ norm squared of the out-degree sequence of $D$, denoted by $\|d^+_D\|_2^2$, is defined as $\|d^+_D\|_2^2 = \sum_{v \in V(D)} d^+_D(v)^2$.
The $L_2$ Tur\'an number for $H$, denoted $\text{ex}^+_{L_2}(m, H)$, is the maximum possible $L_2$ norm squared of the out-degree sequence over all $H$-free digraphs on $m$ vertices:
\[ \text{ex}^+_{L_2}(m, H) = \max \left\{ \sum_{v \in V(D)} d^+_D(v)^2 : \ |V(D)| = m, D \text{ is } H\text{-free} \right\}. \]

The complete $r$-partite directed Tur\'an graph $T_r(m)$ is the complete $r$-partite graph on $m$ vertices with parts of sizes as equal as possible, and where all edges between parts exist in both directions.
There are no edges within parts.

The transitive tournament on $k$ vertices is denoted $TT_k$.
The unique strongly connected self-converse tournament on four vertices is denoted $R_4$, identified by its out-degree sequence $(1,1,2,2)$. The regular tournament on 5 vertices, where every vertex has an out-degree of 2, is denoted $Reg_5$.

\begin{figure}[htbp]
    \centering
    \begin{tikzpicture}[
        vertex/.style={circle, draw, minimum size=20pt, inner sep=2pt, font=\small},
        arc/.style={-Stealth, thick}
        ]
        \begin{scope}[xshift=0cm]
            \node[vertex] (v1) at (0, 2) {$v_1$};
            \node[vertex] (v2) at (2, 2) {$v_2$};
            \node[vertex] (v3) at (2, 0) {$v_3$};
            \node[vertex] (v4) at (0, 0) {$v_4$};
            
            \draw[arc] (v1) -- (v2);
            \draw[arc] (v2) -- (v3);
            \draw[arc] (v3) -- (v4);
            \draw[arc] (v1) -- (v4); 
            
            \draw[arc] (v1) -- (v3);
            \draw[arc] (v2) -- (v4);
            
            \node[font=\bfseries] at (1, -1) {$TT_4$};
        \end{scope}
        
        \begin{scope}[xshift=4cm]
            \node[vertex] (u1) at (0, 2) {$u_1$};
            \node[vertex] (u2) at (2, 2) {$u_2$};
            \node[vertex] (u3) at (2, 0) {$u_3$};
            \node[vertex] (u4) at (0, 0) {$u_4$};
            
            \draw[arc] (u1) -- (u2);
            \draw[arc] (u2) -- (u3);
            \draw[arc] (u3) -- (u4);
            \draw[arc] (u4) -- (u1); 
            
            \draw[arc] (u1) -- (u3);
            \draw[arc] (u2) -- (u4);
            
            \node[font=\bfseries] at (1, -1) {$R_4$};
        \end{scope}

        \begin{scope}[xshift=8.5cm, yshift=1cm]
            \foreach \i in {1,2,3,4,5} {
                \node[vertex] (w\i) at ({90 - (\i-1)*72}:1.4) {$w_\i$};
            }
            
            \foreach \i in {1,2,3,4,5} {
                \pgfmathtruncatemacro{\stepa}{mod(\i,5)+1}
                \pgfmathtruncatemacro{\stepb}{mod(\i+1,5)+1}
                \draw[arc] (w\i) -- (w\stepa);
                \draw[arc] (w\i) -- (w\stepb);
            }
            
            \node[font=\bfseries] at (0, -2) {$Reg_5$};
        \end{scope}
    \end{tikzpicture}
    \caption{The three small tournaments discussed in this paper: the transitive tournament $TT_4$, the strongly connected self-converse tournament $R_4$, and the regular tournament $Reg_5$.}
    \label{fig:tournaments}
\end{figure}
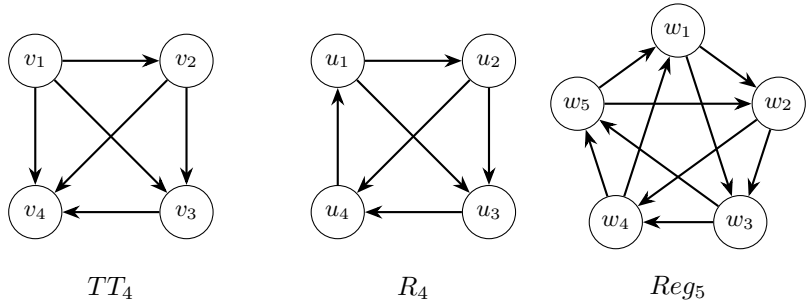

A \emph{palette} is formally defined as a pair $P = (C, A)$, where $C$ is a finite set of colors and $A \subseteq C^3$ is a set of admissible ordered triples.
Intuitively, a palette can be thought of as a set of allowable 3-element sequences of colors, which often represent valid local colorings of edges and vertices in a graph homomorphism context.
An $m$-color palette is a palette where the number of colors is exactly $m$, meaning $|C| = m$.

A \emph{palette homomorphism} from a palette $P = (C, A)$ to a palette $P' = (C', A')$ is a map $f: C \to C'$ that preserves admissible triples; that is, if $(a, b, c) \in A$, then $(f(a), f(b), f(c)) \in A'$.
A palette $P$ is said to \emph{avoid} a palette $P'$ if there is no palette homomorphism from $P'$ to $P$.
The \emph{density} of an $m$-color palette $P = (C, A)$ is the proportion of all possible triples that are admissible, defined as $d(P) = \frac{|A|}{|C|^3}$.

For a loopless digraph $D$, the \emph{left palette} generated by $D$ is defined by introducing a unique color $c_{uv}$ for each arc $uv \in A(D)$, setting the color set to $V(D) \cup A(D)$, and defining the admissible triples as $P_L^D = (V(D) \cup A(D), \{(u, v, c_{uv}) : uv \in A(D)\})$.
Similarly, the \emph{right palette} is $P_R^D = (V(D) \cup A(D), \{(c_{uv}, u, v) : uv \in A(D)\})$.
For the directed cycle of length 3, $\vec{C}_3$, its left palette $P_L$ and right palette $P_R$ are shorthand for $P_L^{\vec{C}_3}$ and $P_R^{\vec{C}_3}$, respectively.
To avoid $P_L$ and $P_R$, a palette $P$ must not contain any subset of colors and triples that maps to the cyclic structure of $\vec{C}_3$. Specifically, avoiding $P_L$ means $P$ cannot contain any six colors (not necessarily distinct) that form three admissible triples corresponding to the vertices and arcs of a directed triangle.

To analyze the structure of a palette $P = (C, A)$, we construct two \emph{auxiliary digraphs}, $G_L(P)$ and $G_R(P)$, both on the vertex set $C$.
These digraphs capture the pairwise dependencies within the admissible triples.
Explicitly, the left auxiliary digraph $G_L(P)$ is generated by the first two coordinates of the triples: an arc $xy \in A(G_L(P))$ exists if and only if there is some color $z \in C$ such that $(x, y, z) \in A$.
Correspondingly, the right auxiliary digraph $G_R(P)$ is generated by the last two coordinates: an arc $yz \in A(G_R(P))$ exists if and only if there is some color $x \in C$ such that $(x, y, z) \in A$.

The density of the palette $P$ is intrinsically bounded by the degrees of these auxiliary digraphs.
Every admissible triple $(x, y, z) \in A$ requires the existence of the arc $xy$ in $G_L(P)$ and the arc $yz$ in $G_R(P)$.
For any fixed middle color $y \in C$, the number of valid choices for the first color $x$ is at most the in-degree of $y$ in $G_L(P)$, denoted $d^-_{G_L}(y)$, and the number of valid choices for the third color $z$ is at most the out-degree of $y$ in $G_R(P)$, denoted $d^+_{G_R}(y)$.
Therefore, the total number of admissible triples $|A|$ is bounded by $\sum_{y \in C} d^-_{G_L}(y) d^+_{G_R}(y)$.
This structural connection allows one to bound the density of the palette $d(P)$ using the sum of products of degrees in its auxiliary digraphs.

The ordered digon-chain on $n$ vertices, denoted $\vec{F}_{n,2}$, is defined by taking $n = 2q + r$ ($0 \le r < 2$) vertices, partitioned into $q$ complete digraph blocks (digons) of size 2, and one final block of size $r$.
Every inter-block arc is oriented from earlier blocks to later blocks.

The edit distance between two labeled digraphs on $n$ vertices is the minimum number of edge additions and deletions required to make them isomorphic.

\section{Main Results}

In this section, we present the exact $L_2$ Tur\'an numbers for small self-converse tournaments and a stability theorem for $\vec{C}_3$-free digraphs.

\subsection{Exact Results for Small Tournaments}

For the transitive tournament $TT_4$, we show that the complete directed 3-partite Tur\'an graph exactly maximizes the objective.

\begin{theorem}\label{thm:tt4}
For all $m \ge 1$, the maximum possible $L_2$ norm squared of the out-degree sequence $\text{ex}^+_{L_2}(m, TT_4)$ is exactly achieved by $T_3(m)$.
That is, $\text{ex}^+_{L_2}(m, TT_4) = \|d^+_{T_3(m)}\|_2^2$.
\end{theorem}

For the strongly connected tournament $R_4$, we established exact values for small $m$ computationally, which led to a precise structural characterization for all $m$.
Let $E(m)$ be defined as uniquely composed of a complete 3-partite directed graph $T_3(k)$ and a transitive tournament $TT_{m-k}$, with all edges between $T_3(k)$ and $TT_{m-k}$ directed from $T_3(k)$ to $TT_{m-k}$, where $k = m$ if $m \equiv 0 \pmod 3$ or $m \equiv 2 \pmod 3$, and $k = m - 1$ if $m \equiv 1 \pmod 3$.

\begin{theorem}\label{thm:r4}
    For all $m \ge 1$, the $R_4$-free digraph on $m$ vertices that maximizes the $L_2$ norm squared of the out-degree sequence is uniquely isomorphic to $E(m)$.
    Consequently, letting $m = 3q + r$, where $r \in \{0, 1, 2\}$ and $q \ge 0$, the exact value of $\text{ex}^+_{L_2}(m, R_4)$ is given by:
    \begin{itemize}
        \item If $r = 0$: $\text{ex}^+_{L_2}(m, R_4) = 12q^3$.
        \item If $r = 1$: $\text{ex}^+_{L_2}(m, R_4) = 12q^3 + 12q^2 + 3q$.
        \item If $r = 2$: $\text{ex}^+_{L_2}(m, R_4) = 12q^3 + 24q^2 + 14q + 2$.
    \end{itemize}
\end{theorem}

Based on exact values computed for small $m$, we also propose a conjecture for the regular tournament $Reg_5$.

\begin{conjecture}
For all $m \ge 1$, the maximum $L_2$ norm squared of the out-degree sequence for a $Reg_5$-free digraph is exactly achieved by the Tur\'an graph $T_4(m)$. That is, $\text{ex}^+_{L_2}(m, Reg_5) = \|d^+_{T_4(m)}\|_2^2$.
Specifically, writing $m = 4q + r$ with $0 \le r < 4$, the exact algebraic formula is:
\begin{itemize}
    \item If $r = 0$: $\text{ex}^+_{L_2}(m, Reg_5) = 36q^3$.
    \item If $r = 1$: $\text{ex}^+_{L_2}(m, Reg_5) = 36q^3 + 27q^2 + 3q$.
    \item If $r = 2$: $\text{ex}^+_{L_2}(m, Reg_5) = 36q^3 + 54q^2 + 22q + 2$.
    \item If $r = 3$: $\text{ex}^+_{L_2}(m, Reg_5) = 36q^3 + 81q^2 + 57q + 12$.
\end{itemize}
\end{conjecture}

\subsection{Stability for $\vec{C}_3$-free Digraphs}

We resolve the stability version of the $L_2$ Tur\'an problem for $\vec{C}_3$-free digraphs.

\begin{theorem} \label{thm:stability}
Let $D = (V, E)$ be a loopless $\vec{C}_3$-free digraph on $n$ vertices.
If $\sum_{v \in V} d^+(v)^2 \ge \frac{n^3}{3} - \delta n^3$,
then the edit distance from $D$ to the ordered digon-chain $\vec{F}_{n,2}$ is at most $O(\delta^{1/2}) n^2$.
\end{theorem}

There is a stability version of the palette density bounds.

\begin{corollary} \label{cor:stability}
If an $m$-color palette avoiding $P_L$ and $P_R$ has density at least $\frac{1}{3} - \frac{1}{3m^2} - \delta$,
then its auxiliary digraphs $G_L$ and $G_R$ have edit distance at most $O(\delta^{1/2}) m^2$ to the ordered digon-chain extremal construction $\vec{F}_{m,2}$.
\end{corollary}

\section{Proof of Theorem \ref{thm:tt4}}

\begin{lemma} \label{lem:tt3_free}
Let $S$ be a $TT_3$-free digraph on $n$ vertices.
Let $a$ be the maximum out-degree among all vertices in $S$.
Let $b = n - a$.
Then the sum of out-degrees $E(S)$ and the $L_2$ norm squared of the out-degree sequence $\|d^+_S\|_2^2$ in $S$ satisfy:
\begin{enumerate}
\item $E(S) \le 2ab$
\item $\|d^+_S\|_2^2 \le ab(a+b)$
\end{enumerate}
\end{lemma}

\begin{proof}
Let w $\in V(S)$ be a vertex achieving the maximum out-degree $d^+_S(w) = a$.
Define its out-neighborhood as $S' = N^+_S(w)$, so $|S'| = a$.
Since $S$ is $TT_3$-free, $S'$ must be an independent set.
If there were any directed edge $x \to y$ within $S'$, then the vertices $\{w, x, y\}$ would induce the edges $w \to x, w \to y$, and $x \to y$, which exactly forms a $TT_3$, a contradiction.
Because $S'$ contains no edges, all out-edges from vertices in $S'$ must go to $V(S) \setminus S'$.
\begin{itemize}
\item For any $x \in S'$, its out-degree in $S$ is at most $|V(S) \setminus S'| = n - a = b$.
There are $a$ such vertices.
\item For any $y \in V(S) \setminus S'$, its out-degree in $S$ is bounded by the global maximum out-degree $a$.
There are $n - a = b$ such vertices.
\end{itemize}

Summing the out-degrees:
\[
E(S) = \sum_{x \in S'} d^+_S(x) + \sum_{y \in V(S) \setminus S'} d^+_S(y) \le a(b) + b(a) = 2ab
\]
Summing the squares:
\[
\|d^+_S\|_2^2 = \sum_{x \in S'} (d^+_S(x))^2 + \sum_{y \in V(S) \setminus S'} (d^+_S(y))^2 \le a(b^2) + b(a^2) = ab(a+b). 
\]
\end{proof}

\begin{proof}[Proof of Theorem \ref{thm:tt4}]
To prove that the maximum possible $L_2$ norm squared of the out-degree sequence for any $TT_4$-free digraph on $m$ vertices is achieved by the complete directed 3-partite Tur\'an graph $T_3(m)$,
we will establish both a matching lower and upper bound.

\paragraph{Step 1: Lower Bound ($T_3(m)$ is $TT_4$-free)}
The complete directed 3-partite Tur\'an graph $T_3(m)$ partitions its $m$ vertices into three independent sets $V_1, V_2, V_3$ of sizes $n_1, n_2, n_3$ that are as equal as possible ($|n_i - n_j| \le 1$ and $n_1+n_2+n_3=m$).
Between any two vertices in different parts, edges exist in both directions (2-cycles).
No edges exist within any part.

A tournament on 4 vertices, $TT_4$, requires exactly one directed edge between every pair of its 4 vertices.
By the Pigeonhole Principle, any set of 4 vertices chosen from $T_3(m)$ must contain at least two vertices in the same part $V_i$.
Because $V_i$ is an independent set, there are zero edges between these two vertices.
Thus, they cannot form a tournament.
By consequence, $T_3(m)$ is $TT_4$-free, giving us a valid lower bound:
\[
\text{ex}^+_{L_2}(m, TT_4) \ge \|d^+_{T_3(m)}\|_2^2 = \sum_{i=1}^3 n_i(m-n_i)^2
\]

\paragraph{Step 2: Exact Upper Bound for $TT_4$-Free Digraphs}
Let $D$ be a $TT_4$-free digraph on $m$ vertices.
Let $\Delta$ be the maximum out-degree in $D$, achieved by some vertex v.
Define $S = N^+_D(v)$, meaning $|S| = \Delta$.

\begin{claim}
The induced subdigraph $D[S]$ is $TT_3$-free.
\end{claim}

\begin{proof}
If $D[S]$ contained a $TT_3$, the vertex $v$ (which points to all vertices in $S$) alongside this $TT_3$ would form a $TT_4$.
This contradicts $D$ being $TT_4$-free.
\end{proof}

We decompose the $L_2$ norm squared of the out-degree sequence of $D$, $\|d^+_D\|_2^2$, over $V(D) \setminus S$ and $S$:
\[
\|d^+_D\|_2^2 = \sum_{u \notin S} (d^+_D(u))^2 + \sum_{u \in S} (d^+_D(u))^2
\]

\begin{enumerate}
\item For $u \notin S$, the out-degree is bounded by the maximum $\Delta$.
With $m - \Delta$ such vertices:
\[
\sum_{u \notin S} (d^+_D(u))^2 \le (m - \Delta)\Delta^2
\]
\item For $u \in S$, its out-edges partition into $S$ and $V(D) \setminus S$.
The number of out-edges to $V(D) \setminus S$ is at most $|V(D) \setminus S| = m - \Delta$.
Thus, $d^+_D(u) \le d^+_S(u) + m - \Delta$.
Squaring this yields:
\[
(d^+_D(u))^2 \le (d^+_S(u))^2 + 2(m - \Delta)d^+_S(u) + (m - \Delta)^2
\]
Summing this inequality over all $\Delta$ vertices in $S$ gives:
\[
\sum_{u \in S} (d^+_D(u))^2 \le \|d^+_S\|_2^2 + 2(m - \Delta)E(S) + \Delta(m - \Delta)^2
\]
\end{enumerate}

Because $S$ is $TT_3$-free, we apply Lemma \ref{lem:tt3_free}.
Let $a$ be the maximum out-degree in $S$, and let $b = \Delta - a$.
Both $a$ and $b$ are non-negative integers.
We substitute $E(S) \le 2ab$ and $\|d^+_S\|_2^2 \le ab\Delta = ab(a+b)$:
\[
\sum_{u \in S} (d^+_D(u))^2 \le ab(a+b) + 4ab(m - \Delta) + \Delta(m - \Delta)^2
\]

Adding the bound for $V(D) \setminus S$, the total $L_2$ norm squared is:
\[
\|d^+_D\|_2^2 \le (m - \Delta)\Delta^2 + ab(a+b) + 4ab(m - \Delta) + \Delta(m - \Delta)^2
\]

Let $c = m - \Delta$.
Notice that $a, B, c$ are non-negative integers summing to $m$.
Substituting $\Delta = a+b$ and $c$:
\[
\|d^+_D\|_2^2 \le c(a+b)^2 + ab(a+b) + 4abc + (a+b)c^2
\]

Expanding this expression algebraically:
\[
c(a^2+2ab+b^2) + a^2b+ab^2 + 4abc + ac^2+bc^2 = a^2b + ab^2 + b^2c + bc^2 + c^2a + ca^2 + 6abc
\]
Notice that this uniquely factorizes into the exact $L_2$ norm squared for a 3-partite graph $F(a,b,c)$:
\[
F(a,b,c) = a(b+c)^2 + b(a+c)^2 + c(a+b)^2 = a^2b + ab^2 + b^2c + bc^2 + c^2a + ca^2 + 6abc
\]
Thus, we have shown that $\|d^+_D\|_2^2 \le F(a,b,c) $ for some partition $a+b+c=m$.

\paragraph{Step 3: Optimization Over Integer Partitions}
To maximize $F(a,b,c)$ over integers $a+b+c=m$, we rewrite it using elementary symmetric polynomials:
\[
F(a,b,c) = m(ab+bc+ca) + 3abc
\]

Suppose the partition is not as balanced as possible.
Then there exist two parts, say $a$ and $b$, such that $a \ge b + 2$.
Consider moving them closer by 1: $a' = a - 1$ and $b' = b + 1$ (leaving $c' = c$).
The net change is:
\[
\Delta F = F(a-1, b+1, c) - F(a,b,c)
\]
\[
= \big[m((a-1)(b+1) + (b+1)c + c(a-1)) + 3(a-1)(b+1)c\big] - \big[m(ab+bc+ca) + 3abc\big]
\]
\[
= m(a - b - 1) + 3c(a - b - 1) = (m + 3c)(a - b - 1)
\]

Since $a \ge b + 2$, we have $a - b - 1 \ge 1$.
Since $m \ge 1$ and $c \ge 0$, we have $m + 3c \ge 1$.
Therefore, $\Delta F \ge 1 > 0$ strictly.

Because balancing any two disparate parts strictly increases the $L_2$ norm squared, the unique global maximum is attained when no two parts differ by more than 1.
These balanced dimensions define the complete directed Tur\'an graph $T_3(m)$.
Thus:
\[
\|d^+_D\|_2^2 \le \max_{a+b+c=m} F(a,b,c) = \|d^+_{T_3(m)}\|_2^2
\]

Combined with our lower bound, the maximum is achieved by $T_3(m)$.
\end{proof}

\section{Proof of Theorem \ref{thm:r4}}

\begin{proof}
    Let $D$ be an optimal $R_4$-free digraph on $m$ vertices that maximizes $\|d^+_D\|_2^2 = \sum_{v \in V(D)} d^+(v)^2$.

    \textbf{Part 1: Symmetrization and the Semicomplete Quotient.}
    Define non-adjacency $u \not\sim v$ if there are zero directed edges between $u$ and $v$ in either direction ($A_{uv} = A_{vu} = 0$). Because $R_4$ is a tournament, any $R_4$ in a digraph could never contain a non-adjacent pair. Thus, the operation of replacing $u$ with a "clone" of $v$ (assigning $u$ identical in- and out-neighborhoods as $v$) preserves $R_4$-freeness.
    
    Let $\Delta_v$ be the change in $\|d^+_D\|_2^2$ when cloning $v$ onto $u$, and $\Delta_u$ the change when cloning $u$ onto $v$. Summing the changes yields:
    \[ \Delta_v + \Delta_u = d^+(v)^2 - d^+(u)^2 + \sum_{x \neq u,v} \left(2d^+(x)(A_{xv} - A_{xu}) + (A_{xv} - A_{xu})^2\right) \]
    plus the symmetric expression for $u$, which cancels the linear cross-terms, yielding:
    \[ \Delta_v + \Delta_u = 2 \sum_{x \neq u, v} (A_{xv} - A_{xu})^2. \]
    Because $D$ is extremal, neither valid clone operation can strictly increase the $L_2$ norm squared, forcing $\Delta_v \le 0$ and $\Delta_u \le 0$, which mandates $\Delta_v + \Delta_u \le 0$. Since squares are non-negative, we must have $A_{xv} = A_{xu}$ for all $x$. Therefore, non-adjacent vertices must have identical in-neighborhoods. This forces non-adjacency to be a transitive equivalence relation: if $x \not\sim y$ and $y \not\sim z$, then $x \not\sim z$. Otherwise, there is an edge between $x$ and $z$. If $x \to z$, then $x \in In(z) = In(y)$, meaning there is an edge $x \to y$, contradicting $x \not\sim y$. If $z \to x$, then $z \in In(x) = In(y)$, meaning there is an edge $z \to y$, contradicting $y \not\sim z$. Thus, $V(D)$ is partitioned into independent sets $V_1, \dots, V_k$.

    For each part $V_i$, choose a vertex $v_i^* \in V_i$ that maximizes out-degree. Create a new graph $D'$ by replacing the out-neighborhood of every $x \in V_i$ with the out-neighborhood of $v_i^*$. Since any $y \in V_j$ shares identical in-neighborhoods with $v_j^*$, the edges between $V_i$ and $V_j$ in $D'$ are uniformly identical to the edges between $v_i^*$ and $v_j^*$ in $D$. Since changing the out-degree of every $x$ to the maximum $d^+(v_i^*)$ weakly increases the objective, $\|d^+_{D'}\|_2^2 \ge \|d^+_D\|_2^2$. The optimal structure must be a blow-up of the quotient digraph $Q = D'[\{v_1^*, \dots, v_k^*\}]$. Because non-adjacency merged into independent sets, there are zero non-edges between distinct parts $V_i, V_j$. Thus, $Q$ is a semicomplete digraph.

    \textbf{Part 2: Tournament Theory and Macroscopic Blocks.}
    By construction, the quotient digraph $Q$ is an induced subdigraph of $D$. Since $D$ is $R_4$-free, $Q$ must also be $R_4$-free. Furthermore, because non-adjacency is an equivalence relation, there are no non-edges between distinct parts, meaning $Q$ is a semicomplete digraph.

    We now analyze the Strongly Connected Components (SCCs) of $Q$. Camion's Theorem \cite{camion1959} states that every strongly connected semicomplete digraph contains a Hamiltonian cycle. Suppose for the sake of contradiction that an SCC of $Q$, say $C$, has $n \ge 4$ vertices. By Camion's Theorem \cite{camion1959}, $C$ contains a Hamiltonian cycle. We can extract a spanning tournament $T$ of $C$ by retaining the edges of this Hamiltonian cycle and, for any remaining 2-cycles, arbitrarily retaining exactly one directed edge while deleting the other. Because $T$ contains a Hamiltonian cycle, it remains strongly connected. Moon's Theorem \cite{moon1966} establishes that every strongly connected tournament on $n \ge 4$ vertices contains a strongly connected subtournament on exactly $4$ vertices. Since $T$ is a strongly connected tournament with at least $4$ vertices, it must contain a strongly connected subtournament of size $4$. The unique strongly connected tournament on $4$ vertices is $R_4$. Consequently, $T$, and by extension $Q$, must contain an $R_4$. This directly contradicts the fact that $Q$ is $R_4$-free. Thus, we deduce that no SCC of $Q$ can have size $\ge 4$.

    Next, we consider the global structure of $Q$. The condensation of any directed graph (the graph of its SCCs) forms a Directed Acyclic Graph (DAG). Because $Q$ is semicomplete, for any two distinct SCCs, say $C_i$ and $C_j$, there must be at least one directed edge between them. Since they are distinct SCCs, there cannot be directed edges in both directions (otherwise they would merge into a single SCC). Thus, all edges between $C_i$ and $C_j$ for $i < j$ must be directed in the exact same direction. This implies that the condensation of $Q$ is a tournament. Since it is also acyclic, it must be a transitive tournament. This forces the SCCs of $Q$ to form an acyclic transitive linear order $C_1 \to C_2 \to \dots \to C_p$, where all edges between $C_i$ and $C_j$ for $i < j$ are directed from $C_i$ to $C_j$.

    Translating this structure back to the blow-up $D'$, the graph is partitioned into a sequence of macroscopic blocks $U_1 \to U_2 \to \dots \to U_p$, where each block $U_i$ is a blow-up of the corresponding SCC $C_i$. Any strongly connected subdigraph of $D'$ must be entirely contained within a single block $U_i$, because all inter-block edges are forward-directed. In particular, any potential $R_4$ must be confined to a single block. Since $|V(C_i)| \le 3$, each block $U_i$ consists of at most $3$ independent sets (parts). Since any tournament in $U_i$ can have at most one vertex per part, the largest tournament in $U_i$ has size at most $3$. Therefore, it is impossible to form an $R_4$ anywhere in $D'$.

    To maximize the $L_2$ norm squared of the out-degree sequence $\|d^+_{D'}\|_2^2$, every permissible edge between distinct parts inside the blocks should be present in both directions. Adding 2-cycles between parts does not introduce any larger tournaments, and thus cannot create an $R_4$. Consequently, to maximize the objective, each $C_i$ must be a complete graph of 2-cycles, which means each block $U_i$ must be a complete multipartite graph of 2-cycles with at most $3$ parts.

    Let the sizes of the parts of $U_i$ be $a, b, c$, with $a+b+c = u_i$, and let $W$ be the total number of vertices in all topologically downstream blocks. The $L_2$ norm squared of the out-degree sequence for the vertices in $U_i$ is given by $f(a,b,c) = a(u_i - a + W)^2 + b(u_i - b + W)^2 + c(u_i - c + W)^2$. Transferring a single vertex from a larger part to a smaller part strictly increases this sum. Therefore, $f$ is uniquely maximized when $a, b, c$ are as equal as possible, optimizing each block to a symmetric $T_3(u_i)$.

    \textbf{Part 3: Exact Algebraic Sequence Optimization.}
    We now optimize the sequence of sizes $u_i$. Consider merging two adjacent blocks $U$ (of size $u$) and $V$ (of size $v$), which currently have all edges directed from $U$ to $V$, into a single block $T_3(u+v)$. Let $W$ be the total number of vertices in all blocks topologically downstream of $V$.

    To understand the change in the objective function, we analyze the sum of squared out-degrees before and after the merge. For a given complete directed 3-partite block $T_3(k)$ of size $k$, its parts are as equal as possible, meaning their sizes are $\lfloor k/3 \rfloor$ and $\lceil k/3 \rceil$. Because each vertex connects in both directions to all vertices outside its own part, its internal out-degree depends solely on the size of its part. A vertex in a part of size $s$ has out-degree $k-s$. Consequently, the internal out-degrees of the vertices in $T_3(k)$ are exactly $\lceil 2k/3 \rceil$ (for vertices in the smaller parts of size $\lfloor k/3 \rfloor$) and $\lfloor 2k/3 \rfloor$ (for vertices in the larger parts of size $\lceil k/3 \rceil$).
    
    Let $r = k \bmod 3$. There are $r$ parts of size $\lceil k/3 \rceil$, containing a total of $r \lceil k/3 \rceil$ vertices, and $3-r$ parts of size $\lfloor k/3 \rfloor$, containing $(3-r) \lfloor k/3 \rfloor$ vertices. Using these exact out-degrees, we directly express the total number of internal directed edges $e(k)$ and the internal sum of squares $S(k) = \|d^+_{T_3(k)}\|_2^2$ as:
    \begin{align*}
        e(k) &= r \lceil k/3 \rceil \lfloor 2k/3 \rfloor + (3-r) \lfloor k/3 \rfloor \lceil 2k/3 \rceil, \\
        S(k) &= r \lceil k/3 \rceil \lfloor 2k/3 \rfloor^2 + (3-r) \lfloor k/3 \rfloor \lceil 2k/3 \rceil^2.
    \end{align*}
    By substituting $k=3q+r$, one can algebraically simplify the edge count to exactly $e(k) = 2 \lfloor k^2/3 \rfloor$. We will use $\|d^+_{T_3(k)}\|_2^2$ to denote this internal sum of squares $S(k)$.

    Before the merge, a vertex $x \in U$ has an out-degree composed of its internal edges in $U$, edges to all $v$ vertices in $V$, and edges to the $W$ downstream vertices. Thus, its total out-degree is $d^+_U(x) + v + W$. Similarly, a vertex $y \in V$ has out-degree $d^+_V(y) + W$. Summing the squares over $U$ and $V$ yields the pre-merge sum:
    \begin{align*}
        S_{\text{before}} &= \sum_{x \in U} (d^+_U(x) + v + W)^2 + \sum_{y \in V} (d^+_V(y) + W)^2 \\
        &= \|d^+_{T_3(u)}\|_2^2 + 2(v+W)e(u) + u(v+W)^2 + \|d^+_{T_3(v)}\|_2^2 + 2We(v) + vW^2,
    \end{align*}
    where we used the fact that the sum of out-degrees within $U$ is exactly the total number of internal edges $e(u)$, and similarly for $V$.

    After merging $U$ and $V$ into a single $T_3(u+v)$ block, every vertex $z$ in the merged block has its internal out-degree plus edges to the $W$ downstream vertices, giving $d^+_{T_3(u+v)}(z) + W$. The post-merge sum is:
    \begin{align*}
        S_{\text{after}} &= \sum_{z \in U \cup V} (d^+_{T_3(u+v)}(z) + W)^2 \\
        &= \|d^+_{T_3(u+v)}\|_2^2 + 2We(u+v) + (u+v)W^2.
    \end{align*}

    The net change in the $L_2$ norm squared is $\Delta(u, v, W) = S_{\text{after}} - S_{\text{before}}$. Notice that the $W^2(u+v)$ terms cancel out exactly. Grouping the remaining terms by their dependence on $W$, we obtain $\Delta(u, v, W) = \Delta_0(u, v) + W \cdot \Delta_W(u, v)$, where:
    \begin{align*}
        \Delta_W(u, v) &= 2(e(u+v) - e(u) - e(v)) - 2uv, \\
        \Delta_0(u, v) &= \|d^+_{T_3(u+v)}\|_2^2 - \|d^+_{T_3(u)}\|_2^2 - \|d^+_{T_3(v)}\|_2^2 - 2v \cdot e(u) - u v^2.
    \end{align*}
    By expanding the exact edge formulas, one finds $\Delta_W(u, v) =$ $4\lceil 2uv/3 \rceil - 2uv$, which evaluates to strictly positive values for all integers $u,v \ge 1$ (e.g., exact $\Delta_W(1,1) = 2 > 0$ and $\Delta_W(1,2) = 4 > 0$). Thus, any downstream vertices ($W \ge 1$) incentivize merging.

    We evaluate the base merge difference $\Delta_0(u, v)$ for pairs:
    If $v \ge 2$, merging strictly increases the $L_2$ norm squared.
    If $v = 1$:
    \begin{itemize}
        \item For $u = 3q$: $\Delta_0(3q, 1) = \|d^+_{T_3(3q+1)}\|_2^2 - \|d^+_{T_3(3q)}\|_2^2 - 2e(3q) - 3q = -q \le 0$ (since $u \ge 1$, $q$ must be an integer $q \ge 1$, which trivially enforces $-q < 0$).
        \item For $u = 3q+1$: $\Delta_0(3q+1, 1) = \|d^+_{T_3(3q+2)}\|_2^2 - \|d^+_{T_3(3q+1)}\|_2^2 - 2e(3q+1) - (3q+1) = q + 1 > 0$.
        \item For $u = 3q+2$: $\Delta_0(3q+2, 1) = \|d^+_{T_3(3q+3)}\|_2^2 - \|d^+_{T_3(3q+2)}\|_2^2 - 2e(3q+2) - (3q+2) = 3q + 4 > 0$.
    \end{itemize}
    Thus, merging strictly increases the $L_2$ norm squared in all cases except when breaking $3q+1$ into $T_3(3q) \to TT_1$ at the absolute topological end of the sequence ($W=0$). Furthermore, if $W \ge 1$, the strictly positive $\Delta_W$ dictates $\Delta(3q, 1, W \ge 1) \ge -q + 2q(W) > 0$. Therefore, the global maximum is uniquely found by condensing all blocks entirely, breaking off a final sink exclusively when $m \equiv 1 \pmod 3$.

    \textbf{Part 4: Uniqueness and Conclusion.}
    We showed that $\|d^+_{D'}\|_2^2 \ge \|d^+_D\|_2^2$, and $E(m)$ uniquely maximizes $D'$. For $D$ to be optimal, we must have $\|d^+_{D'}\|_2^2 = \|d^+_D\|_2^2$. The out-degree of any $u \in V_i$ in $D$ satisfies $d^+_D(u) \le d^+_D(v_i^*) = d^+_{D'}(u)$. Equality of the $L_2$ norm squared implies $d^+_D(u) = d^+_{D'}(u)$ for all $u$. If the out-degree is zero (e.g., in the sink part), then $Out_D(u) = \emptyset = Out_{D'}(u)$ trivially. For all non-sink parts in $E(m)$, every vertex $v_i^*$ has the maximum possible number of out-edges permitted by the independent set structure of $V_i$: it connects to every vertex outside its equivalence class. Thus, to achieve $d^+_D(u) = d^+_{D'}(u)$, $u$ must also possess all these inter-class edges. $u$ cannot add any edges within $V_i$, as $V_i$ is an independent set. Therefore, the out-neighborhoods must be identical: $Out_D(u) = Out_{D'}(u)$ for all $u$. This implies that $D$ and $D'$ are isomorphic, and since $E(m)$ uniquely maximizes $D'$, $D$ is uniquely isomorphic to $E(m)$. 
    Calculating the $L_2$ norm squared of the out-degree sequence for $E(m)$ yields the exact formulas.
\end{proof}

\section{Proof of Theorem \ref{thm:stability}}

Here we provide the detailed algebraic proof for Theorem \ref{thm:stability}.

\begin{proof}
To address the stability of $D$ without making any localized assumptions about component sizes,
we apply an exact global structural majorization argument on its out-degree sequence.

We first formally establish that any loopless $\vec{C}_3$-free digraph on $k$ vertices contains at most $k^2/2$ edges Maximum edges in $\vec{C}_3$-free digraphs.
We proceed by induction on $k$.
The base cases $k=1, 2$ trivially hold.
Assume the bound holds for all $\vec{C}_3$-free digraphs on strictly fewer than $k$ vertices.
If the digraph $G$ has no digons (directed 2-cycles),
it is an oriented graph with at most $\binom{k}{2} < k^2/2$ edges.
If $G$ contains a digon on vertices $\{u, v\}$,
then for any other vertex $w$,
there are at most 2 directed edges between $w$ and $\{u,v\}$.
If there were 3 or more edges,
by the Pigeonhole Principle,
$w$ must have at least 2 edges with one of the vertices (say, $u$),
forcing a digon $w \leftrightarrow u$.
The third edge must connect $w$ and $v$.
If it is $w \to v$,
then $w \to v \to u \to w$ is a directed triangle ($\vec{C}_3$).
If it is $v \to w$,
then $v \to w \to u \to v$ forms a $\vec{C}_3$.
Both cases contradict $G$ being $\vec{C}_3$-free.
Summing the edges of $G$ gives:
at most $(k-2)^2/2$ edges in $G \setminus \{u,v\}$ (by the inductive hypothesis),
plus at most $2(k-2)$ cross-edges,
plus $2$ edges for the internal digon $\{u,v\}$.
This totals $\frac{(k-2)^2}{2} + 2(k-2) + 2 = \frac{k^2}{2}$ edges,
completing the induction.

Let $D = (V,E)$ be a loopless $\vec{C}_3$-free digraph on $n$ vertices.
Order the vertices $v_1, \dots, v_n$ such that their out-degrees are non-decreasing:
$d^+(v_1) \le d^+(v_2) \le \dots \le d^+(v_n)$.
We define $x_i = d^+(v_{n-i+1})$,
creating the non-increasing sequence $x_1 \ge x_2 \ge \dots \ge x_n \ge 0$.

Let $V_k = \{v_{n-k+1}, \dots, v_n\}$ be the subset of the $k$ vertices with the highest out-degrees,
and let $S_k = \sum_{i=1}^k x_i$.
The sum $S_k$ exactly evaluates the number of outgoing edges from $V_k$.
Thus, $S_k = e(V_k) + e(V_k, V \setminus V_k)$.
By Step 1, the induced edges satisfy $e(V_k) \le k^2/2$.
Since $D$ is simple, the number of forward edges satisfies $e(V_k, V \setminus V_k) \le k(n-k)$.
Therefore:
\begin{equation}
S_k \le \frac{k^2}{2} + k(n-k) = kn - \frac{k^2}{2}.
\end{equation}
We define the upper bound $M_k = kn - k^2/2$ and the sequence of non-negative structural deficits $D_k = M_k - S_k \ge 0$.

We algebraically expand the global $L_2$ norm squared of the out-degree sequence via Abel summation.
Taking $x_{n+1} = 0$, we have:
\begin{equation} \label{eq:sum_x2}
\sum_{k=1}^n x_k^2 = \sum_{k=1}^n S_k (x_k - x_{k+1}).
\end{equation}
Substituting $S_k = M_k - D_k$ into \eqref{eq:sum_x2} yields:
\begin{equation} \label{eq:sub_Sk}
\sum_{k=1}^n x_k^2 = \sum_{k=1}^n M_k (x_k - x_{k+1}) - \sum_{k=1}^n D_k (x_k - x_{k+1}).
\end{equation}
Applying reverse summation by parts to the first term (with $M_0 = 0$) gives $\sum_{k=1}^n (M_k - M_{k-1}) x_k$.
Let $y_k = M_k - M_{k-1} = n - k + 1/2$.
Applying summation by parts a second time to $\sum_{k=1}^n y_k x_k$:
\begin{equation}
\sum_{k=1}^n y_k x_k = y_n S_n + \sum_{k=1}^{n-1} S_k (y_k - y_{k+1}).
\end{equation}
Observe that $y_n = 1/2$ and $y_k - y_{k+1} = 1$.
Substituting $S_k = M_k - D_k$ again yields:
\begin{equation} \label{eq:yk_substitution}
\sum_{k=1}^n y_k x_k = \left( \frac{1}{2} M_n + \sum_{k=1}^{n-1} M_k \right) - \left( \frac{1}{2} D_n + \sum_{k=1}^{n-1} D_k \right).
\end{equation}
The sum involving purely $M_k$ analytically collapses back to $\sum_{k=1}^n y_k^2$.
Specifically, applying the exact same summation by parts identity to the sequence $x_k = y_k$ (where its partial sum is $S_k = M_k$) directly yields $\sum_{k=1}^n y_k^2 = \frac{1}{2} M_n + \sum_{k=1}^{n-1} M_k$ because $y_n = 1/2$ and $y_k - y_{k+1} = 1$.
We can explicitly evaluate this absolute maximum:
\begin{equation}
\sum_{k=1}^n y_k^2 = \sum_{k=1}^n \left(n - k + \frac{1}{2}\right)^2 = \sum_{j=1}^n \left(j - \frac{1}{2}\right)^2 = \frac{n^3}{3} - \frac{n}{12}.
\end{equation}
Plugging this result back into \eqref{eq:sub_Sk}, we derive the exact telescoping identity:
\begin{equation} \label{eq:exact_identity}
\sum_{k=1}^n x_k^2 = \frac{n^3}{3} - \frac{n}{12} - \sum_{k=1}^{n-1} D_k (1 + x_k - x_{k+1}) - D_n \left( \frac{1}{2} + x_n \right).
\end{equation}

By the stability premise of the theorem,
$\sum_{v \in V} d^+(v)^2 = \sum_{k=1}^n x_k^2 \ge \frac{n^3}{3} - \delta n^3$.
Inserting this bounds the aggregate deficit:
\begin{equation}
\sum_{k=1}^{n-1} D_k (1 + x_k - x_{k+1}) + D_n \left( \frac{1}{2} + x_n \right) \le \delta n^3 - \frac{n}{12} \le \delta n^3.
\end{equation}
Because the sequence is non-increasing ($x_k \ge x_{k+1}$) and degrees are non-negative ($x_n \ge 0$),
the coefficients satisfy $1 + x_k - x_{k+1} \ge 1$ and $1/2 + x_n \ge 1/2$.
Dropping the multipliers mathematically ensures:
\begin{equation} \label{eq:sum_D_bound}
\sum_{k=1}^{n-1} D_k + \frac{1}{2} D_n \le \delta n^3.
\end{equation}
Since all $D_k \ge 0$,
this guarantees $\sum_{k=1}^{n-1} D_k \le \delta n^3$ and $D_n \le 2\delta n^3$.
Thus, $\sum_{k=1}^n D_k \le 2\delta n^3$.

We now map this abstract deficit to structural missing edges.
Recall $D_k = \left(\frac{k^2}{2} - e(V_k)\right) + m_k$,
where $m_k = k(n-k) - e(V_k, V \setminus V_k)$ is exactly the number of missing forward edges directed from $V_k$ to $V \setminus V_k$.
Since $e(V_k) \le k^2/2$, we have $m_k \le D_k$,
dictating that $\sum_{k=1}^n m_k \le 2\delta n^3$. 

A forward edge logically points from a vertex of higher index to one of lower index (i.e., $v_a \to v_b$ with $a > b$),
aligning precisely with the topological structure of a transitive tournament where the source possesses the maximum out-degree.
Consider a missing forward edge $v_a \not\to v_b$ ($a > b$).
Vertex $v_a \in V_k$ if and only if $n-a+1 \le k$,
and $v_b \notin V_k$ if and only if $k \le n-b$.
Thus, this missing pair registers in $m_k$ exactly $(n-b) - (n-a+1) + 1 = a - b$ times.
Summing over all missing forward edges bridges this topology directly to our deficit bound:
\begin{equation}
\sum_{a > b} (a - b) \mathbf{1}(v_a \not\to v_b) = \sum_{k=1}^n m_k \le 2\delta n^3.
\end{equation}
Let $M_F$ be the total count of missing forward edges.
Introducing a threshold $\gamma \in (0, 1)$, we split the summation.
By Markov's inequality, missing edges spanning an index length of $a - b \ge \gamma n$ are bounded by $\frac{2\delta n^3}{\gamma n} = \frac{2\delta}{\gamma} n^2$.
The number of total possible index pairs $(a,b)$ satisfying $a > b$ and $a - b < \gamma n$ is bounded by $\gamma n^2$.
Minimizing $M_F \le \frac{2\delta}{\gamma} n^2 + \gamma n^2$ by choosing $\gamma = \sqrt{2\delta}$ bounds $M_F \le 2\sqrt{2\delta}n^2 = O(\delta^{1/2})n^2$.

Finally, let $m = |E| = S_n$ be the total number of edges in $D$.
Because $D_n = \frac{n^2}{2} - S_n$,
we deduce $m = \frac{n^2}{2} - D_n$. 
Separating $m$ into forward edges $F$ and backward edges $B$ ($v_a \to v_b$ where $a < b$),
we recognize $F = \frac{n^2 - n}{2} - M_F$.
Consequently:
\begin{equation}
B = m - F = \frac{n^2}{2} - D_n - \frac{n^2 - n}{2} + M_F \le M_F + \frac{n}{2} = O(\delta^{1/2})n^2.
\end{equation}

To structurally transform $D$ into the ordered transitive tournament $\vec{T}_n$,
we must simply add the $M_F$ missing forward edges and delete the $B$ extraneous backward edges,
requiring an edit distance of at most $M_F + B \le 2M_F + n/2 = O(\delta^{1/2})n^2$. 
Because the ordered digon-chain $\vec{F}_{n,2}$ achieves its precise topological maximum density by augmenting $\vec{T}_n$ with at most $\lfloor n/2 \rfloor$ disjoint backward edges (forming adjacent digons),
expanding $D$ to $\vec{F}_{n,2}$ demands an edit distance of at most $O(\delta^{1/2})n^2 + n/2 = O(\delta^{1/2})n^2$. 
\end{proof}

\begin{proof}[Proof of Corollary \ref{cor:stability}]
The maximum possible density is exactly $\frac{1}{3} - \frac{1}{3m^2}$ Established in Ai et al..
The density of such a palette is bounded by $\frac{1}{m^3} \sum_{v} d^-_{G_L}(v) d^+_{G_R}(v)$,
where $G_L$ and $G_R$ are loopless $\vec{C}_3$-free digraphs on $m$ vertices.

Since $2\sum d^-_{G_L}(v) d^+_{G_R}(v) \le \sum (d^-_{G_L}(v))^2 + \sum (d^+_{G_R}(v))^2$,
the hypothesis implies that both $L_2$ sums must individually exceed $\frac{m^3}{3} - O(\delta) m^3$.
Applying Theorem \ref{thm:stability} bounds the structural deviation of $G_L$ and $G_R$ from $\vec{F}_{m,2}$.
\end{proof}

\section{Methodology}

The results presented in this paper are the product of a human-AI collaborative research process, utilizing the Google DeepMind AI co-mathematician \cite{zheng2026}. The investigation began with extensive computational exploration to compute the maximum sum of squared out-degrees across a comprehensive set of generated digraphs for small cases. By scaling these computations and analyzing the results, the AI co-mathematician and human users formulated precise conjectures mapping the bounds of the degree-square Tur\'an number for $TT_4$ and $R_4$.

Once the computational evidence strongly supported these exact formulas, rigorous mathematical proofs were developed for the extremal formulas for $TT_4$ and $R_4$. Furthermore, we formally established the stability of $\vec{C}_3$-free digraphs. This collaborative approach efficiently bridged computational exploration with formal mathematical verification, providing a comprehensive resolution to these degree-square Tur\'an problems.

\section{Conclusions}

In this work, we resolved multiple exact $L_2$ Tur\'an problems for small self-converse tournaments.
We also established a quantitative stability result for $\vec{C}_3$-free digraphs using global out-degree majorization.
Future research should investigate whether these stability frameworks extend to broader classes of forbidden subdigraphs and whether the exact result for $Reg_5$ can be formally proved using similar algebraic reductions.


\begin{thebibliography}{99}

\bibitem{ai2026finite}
Jiangdong Ai, Bin Chen, Ming Chen, Zilong Yan, and Tianxiao Zhao.
\newblock Finite palette endpoints and $L_2$ Tur\'an problems.
\newblock \emph{arXiv preprint arXiv:2606.03520}, 2026.

\bibitem{brown1970extremal}
W. G. Brown and F. Harary.
\newblock Extremal digraphs.
\newblock In \emph{Combinatorial Theory and Its Applications}, volume 4 of \emph{Colloq. Math. Soc. J\'anos Bolyai}, pages 135--198. North-Holland, 1970.

\bibitem{bucic2023uniform}
M. Bu\'ci\'c, J. W. Cooper, D. Kr\'al', S. Mohr, and D. Munh\'a Correia.
\newblock Uniform Tur\'an density of cycles.
\newblock \emph{Trans. Amer. Math. Soc.}, 376(7):4765--4809, 2023.

\bibitem{camion1959}
P. Camion.
\newblock Chemins et circuits hamiltoniens des graphes complets.
\newblock \emph{C. R. Acad. Sci. Paris}, 249:2151--2152, 1959.

\bibitem{erdos1966limit}
P. Erd\H{o}s and M. Simonovits.
\newblock A limit theorem in graph theory.
\newblock \emph{Studia Sci. Math. Hungar.}, 1:51--57, 1966.

\bibitem{erdos1982ramsey}
P. Erd\H{o}s and V. T. S\'os.
\newblock On Ramsey-Tur\'an type theorems for hypergraphs.
\newblock \emph{Combinatorica}, 2(3):289--295, 1982.

\bibitem{erdos1946structure}
P. Erd\H{o}s and A. H. Stone.
\newblock On the structure of linear graphs.
\newblock \emph{Bull. Amer. Math. Soc.}, 52:1087--1091, 1946.

\bibitem{glebov2016problem}
R. Glebov, D. Kr\'al', and J. Volec.
\newblock A problem of Erd\H{o}s and S\'os on 3-graphs.
\newblock \emph{Israel J. Math.}, 211(1):349--366, 2016.

\bibitem{king2025possible}
D. King, S. Piga, M. Sales, and B. Sch\"ulke.
\newblock On possible uniform Tur\'an densities.
\newblock \emph{arXiv preprint arXiv:2504.21220}, 2025.

\bibitem{kral2025uniform}
D. Kr\'al', F. Ku\v{c}er\'ak, A. Lamaison, and G. Tardos.
\newblock Uniform Tur\'an density---palette classification.
\newblock \emph{arXiv preprint arXiv:2505.17325}, 2025.

\bibitem{lamaison2024palettes}
A. Lamaison.
\newblock Palettes determine uniform Tur\'an density.
\newblock \emph{arXiv preprint arXiv:2408.09643}, 2024.

\bibitem{lamaison2024uniform}
A. Lamaison and Z. Wu.
\newblock The uniform Tur\'an density of large stars.
\newblock \emph{arXiv preprint arXiv:2409.03699}, 2024.

\bibitem{lin2026extremal}
H. Lin, G. Wang, W. Zhou, and Y. Zhou.
\newblock Extremal problems in uniformly dense hypergraphs and digraphs.
\newblock \emph{arXiv preprint arXiv:2603.10766}, 2026.

\bibitem{lin2026uniform}
H. Lin, G. Sun, G. Wang, and W. Zhou.
\newblock Uniform Tur\'an densities of $k$-uniform hypergraphs.
\newblock \emph{arXiv preprint arXiv:2605.15105}, 2026.

\bibitem{lin2025turan}
H. Lin and W. Zhou.
\newblock Tur\'an density of stars in uniformly dense hypergraphs.
\newblock \emph{arXiv preprint arXiv:2510.12576}, 2025.

\bibitem{moon1966}
J. W. Moon.
\newblock On subtournaments of a tournament.
\newblock \emph{Canad. Math. Bull.}, 9(3):297--301, 1966.

\bibitem{reiher2018turan}
C. Reiher, V. R\"odl, and M. Schacht.
\newblock On a Tur\'an problem in weakly quasirandom 3-uniform hypergraphs.
\newblock \emph{J. Eur. Math. Soc.}, 20(5):1139--1159, 2018.

\bibitem{zhou2023turan}
W. Zhou and B. Li.
\newblock The Tur\'an number of directed paths and oriented cycles.
\newblock \emph{Graphs Combin.}, 39:Article 47, 2023.

\bibitem{zheng2026}
D. Zheng et al.
\newblock AI co-mathematician: Accelerating mathematicians with agentic AI.
\newblock \emph{arXiv preprint arXiv:2605.06651}, 2026.

\end{thebibliography}
\end{document}